\documentclass[11pt,reqno]{amsart}

\usepackage{verbatim}
\usepackage{graphicx}
\usepackage{subfigure}
\usepackage{amsfonts}
\usepackage{amssymb}
\usepackage{amsmath}
\usepackage{amsthm}
\usepackage{latexsym}
\usepackage{amscd}
\usepackage{mathrsfs}
\usepackage{verbatim}
\usepackage{mathrsfs}
\usepackage{enumitem}
\usepackage{braket}
\usepackage[margin=1.3in]{geometry}
\usepackage{comment}
\usepackage{bm}

\usepackage{imakeidx}
\makeindex[title=Index of Notation]

\usepackage{hyperref}
\hypersetup{
	colorlinks=true,
	linkcolor=black,
	urlcolor=black,
	citecolor=black
}
\usepackage{caption}
\usepackage{array}

\usepackage{amssymb}
\usepackage{pdfpages}
\usepackage{verbatim}

\usepackage{appendix}
\usepackage{xspace}

\newtheorem{theorem}{Theorem}[section]

\newtheorem{lemma}[theorem]{Lemma}
\newtheorem{proposition}[theorem]{Proposition}
\newtheorem{corollary}[theorem]{Corollary}
\theoremstyle{definition}
\newtheorem{definition}[theorem]{Definition}

\theoremstyle{remark}
\newtheorem{remark}[theorem]{Remark}

\newcommand{\IN}{\mathbb{N}}
\newcommand{\IR}{\mathbb{R}}

\numberwithin{equation}{section}

\makeatletter
\@namedef{subjclassname@2020}{\textup{2020} Mathematics Subject Classification}
\makeatother

\begin{document}
	\title[well-posedness of mean curvature flow]{Well-posedness of mean curvature flow}
	
	\author{Yongheng Han}
	\address{School of Mathematical Science, University of Science and Technology of China,  Hefei City,  Anhui Province 230026}
	
	\email{hyh2804@mail.ustc.edu.cn}
	
	\date{\today}

	
	\begin{abstract}
		In this paper, using heat kernel estimates and contraction mapping principle, we give a new proof of the existence and uniqueness of mean curvature flow starting from hypersurface with bounded second fundamental form. Moreover, we show the continuous dependence of mean curvature flow on initial data.
	\end{abstract}
	
	\maketitle
	\tableofcontents
	\section{Introduction}
	Let $\mathbf{x}_0:M^n\to \mathbb{R}^{n+1}$ be a  hypersurface in $\mathbb{R}^{n+1}$. A one-parameter family of immersions $\mathbf{x}(p,t):M\to  \mathbb{R}^{n+1}$ is called a mean
	curvature flow, if $\mathbf{x}$ satisfies the equation
	\begin{equation}\label{eq0.1}
		\frac{\partial \mathbf{x}}{\partial t}=-H\mathbf{n},\quad \mathbf{x}(0)=\mathbf{x}_0,
	\end{equation}
	where $H$ denotes the mean curvature of the hypersurface $M_t:=\mathbf{x}(t)(M)$ and $\mathbf{n}$ denote  the unit normal vector field of $M_t$.  Huisken \cite{Hui90} proved that if the flow \eqref{eq0.1} develops a singularity at time $T<\infty$, then the second fundamental form will blow up at time $T$. 
	
	The existence of mean curvature flow under various boudary conditions is well-known 
	(cf.\cite{Lib1,Hui89,Lib2,Guan}). Ecker and Husiken \cite{EH} established the short time existence of the mean curvature flow on complete  hypersurface with bounded second fundamental form. Chen and Yin \cite{CY} showed the uniqueness of the mean curvature flow on complete submanifold with bounded second fundamental form.
	
	The problem of mean curvature flow with rough initial data have been extensively studied. In the case that $M$ is a Lipschitz graph, short time existence was proved by Ecker and Huisken in \cite{EH}(see also \cite{Wan04,HL}).  Clutterbuck \cite{Clu} studied the graphic mean curvature flow equation with continuous initial data over convex domain. Hershkovits  
	\cite{Her1,Her2} showed the existence and uniqueness of smooth flows in $\IR^m$ starting from a Reifenberg set.
	
	Similar questions have been studied in other geometric flows. Wang \cite{Wan12} proved the well-posedness for the heat flow of biharmonic maps with initial data in $BMO$  space. Simon \cite{Sim} showed the existence of Ricci flow form a $C^0$ metric. In \cite{Bam14,Bam15}, Bamler showed that hyperbolic manifolds and symmetric spaces are stable under small $C^0$ perturbation. 
	In \cite{Bur19}, Burkhardt-Guim showed the  continuous dependence of closed Ricci-DeTurck flow. Recently, Cai and Wang \cite{CW} demonstrated the  well-posedness of  Ricci-DeTurck flow of noncompact  Riemannian manifold.
	
	In first part of the paper, we use the heat kernel method and the contraction mapping principle to prove the existence and uniqueness of the mean curvature flow with bounded second fundamental form.
	
	\begin{theorem}\label{thm0.1}
		Let $\mathbf{x}_0:M^n\to \IR^{n+1}$ be an isometrically immersed Riemannian manifold with bounded second fundamental form $|A|\leq \kappa$. Then there exists a sufficient small constant $T=T(\kappa,n)>0$ and a  family of immersions $\{\mathbf{x}(\cdot,t)\}_{t\in [0,T]}$ that solves the mean curvature flow with initial data $\mathbf{x}(\cdot,0)=\mathbf{x}_0$, $\mathbf{x}(\cdot ,t)$ is unique in $X_T$(see definition \ref{def3.1}).
	\end{theorem}
	
	In the second part of the paper, we establish the continuous dependence on initial data. For a Lipschitz function $f$ on $M$. We denote $\|f\|_{C^{0,1}(M)}:=\|f\|_{L^\infty(M)}+\|\nabla f\|_{L^\infty(M)}$.
	
	\begin{theorem}\label{thm0.2}
		Let $\mathbf{x}_0:M^n\to \IR^{n+1}$ be an isometrically immersed Riemannian manifold with bounded second fundamental form $|A|\leq \kappa$ and $\mathbf{x}(\cdot,t)$ be the mean curvature flow in Theorem \ref{thm0.1}. There exists a costant $\varepsilon=\varepsilon(n)>0$ and $C=C(n,\kappa),T'=T'(n,\kappa)>0$  such that for any $\|f_0\|_{C^{0,1}}\leq \varepsilon$	there is a mean curvature flow $\tilde{\mathbf{x}}:M\times [0,T']\to  \IR^{n+1}$ with initial data $M_{f_0}$ such that $\tilde{\mathbf{x}}(t)(M)=M_{f(\cdot,t)}$  and
		\begin{equation}
			\|f(\cdot,t)\|_{C^{0,1}(M_t)}\leq C\|f_0\|_{C^{0,1}(M_0)}, \quad \forall\; t\in [0,T'].
		\end{equation}
		The solution is unique in the ball $\{f|\|f\|_{X_{T'}}\leq C\varepsilon\}$.
	\end{theorem}

	\noindent
	\textbf{Acknowledgment:}
	The author would like to thank his advisor Bing Wang for suggesting this problem. The author is supported by  the Project of Stable Support for Youth Team in
	Basic Research Field,  Chinese Academy of Sciences,  (YSBR-001).

	
	
	
	\section{Graphic mean curvature flow equations}
	
	In \cite{CM16,CM19,SX}, they calcuated the expressions of equations of graphic rescaled MCFs. In this section, we calculate the expression of equation of graphic MCFs. Let $M\subset \IR^{n+1}$ be a smooth complete hypersurface. Given a function $u$ on $M$, we define $M_u$  by 
	\begin{equation}
		M_u= \{x+u\mathbf{n}|x\in M\}
	\end{equation}
	where $\mathbf{n}$ is the unit  normal vector of $M$ at $x$.
	
	\subsection{ MCF equation over a manifold}
	\begin{lemma}\cite{CM16}\label{lm1.1}
		There are functions $w,v,\eta$ depending on $(p,s,y)\in M \times\IR\times T_pM$ that are smooth for $|s|$ smooth enough and depend smoothly on $M$ :
		\begin{itemize}
			\item The relative area element $v(p)=\sqrt{\mathrm{det}g^u_{ij}(p)}/\sqrt{\mathrm{det}g_{ij}(p)}$ where $g_{ij}$ is the metric from $M$ and $g^u_{ij} $ is the metric from $M_u$.
			\item The mean curvature $H_u(p)$ of $M_u$ at $p+u(p)$.
			\item The speed function $w_{u}=\langle e_{n+1},\mathbf{n}_u\rangle^{-1} $ evaluated at the point $p+u(p)$.
		\end{itemize}
		The function $w,v$ satisfy:
		\begin{itemize}
			\item $w(p,s,0)\equiv 1,\partial_{s}(p,s)=0,\partial_{y_\alpha}w(p,s,0)=0$ and $\partial_{y_\alpha}\partial_{y_\beta}w(p,0,0)=\delta_{\alpha\beta}. $
			\item $v(p,0,0)=1$; the non-zero first and second order are $\partial_{s}v(p,0,0)=H(p)$,\\ $\partial_s^2v(p,0,0)=H^2(p)-|A|^2(p),\partial_{p_j}\partial_{s}v(p,0,0)=H_j(p)$, and $\partial_{y_\alpha}\partial_{y_\beta}v(p,0,0)=\delta_{\alpha\beta}.$
			
		\end{itemize}
	\end{lemma}

	\begin{corollary}\cite{CM16}\label{coro1.2}
		The mean curvature of $M_{u}$ is 
		\begin{equation}
			H_{u}=\frac{w}{v}[\partial_sv-\mathrm{div}(\partial_{y^\alpha}v)]
		\end{equation}
		where $v$ and its derivatives are evaluated at $(p,u(p),\nabla u(p))$. Furthermore, the linearization of $L$ of $H_{u}$ is 
		\begin{equation}
			\begin{split}
				L:=	-\frac{d}{dt}\bigg|_{t=0}H_{tu}=\Delta u+|A|^2u.
			\end{split}
		\end{equation}
	\end{corollary}
	
	\begin{proposition}\cite{CM16}\label{prop1.3}
		The graph $M_u$ flow by MCF if and only if $u$ satisfies
		\begin{equation}\label{eq1.4}
			\begin{split}
				\partial_tu(p,t)&=-\frac{w^2}{v}[\partial_sv-\mathrm{div}(\partial_{y^\alpha}v)]=-H_0+Lu+\mathcal{Q}(u),
			\end{split}
		\end{equation}
		where $H_0$ is the mean curvature of $M$ and	$\mathcal{Q}(u)$ is quadratic in $u$ i.e, $|\mathcal{Q}(su)|\leq C_us^2$.
	\end{proposition}
	More precisely, we have the following estimates:
	\begin{proposition}\cite{CM19}\label{prop1.4}
		There is $\delta>0$ such that if $\|u\|_{C^{0,1}}<\delta,\|v\|_{C^{0,1}}\leq \delta$, then we have 
		\begin{equation}\label{eq1.5}
			\begin{split}
				|\mathcal{Q}(u)|\leq C(|u|^2+|\nabla u|^2+|\nabla^2 u|(|\nabla u|+|u|)),
			\end{split}
		\end{equation}
		and 
		\begin{equation}\label{eq1.6}
			\begin{split}
				|\mathcal{Q}(u)-\mathcal{Q}(v)|&\leq C(\|u\|_{C^{0,1}}+|\nabla v|_{C^{0,1}})(|u-v|+|\nabla(u-v)|+|\nabla^2 (u-v)|)\\
				&+C(|\nabla (u-v)|+|u-v|)|\nabla^2 u|.
			\end{split}
		\end{equation}
	\end{proposition}	

	\subsection{MCF equation over evolving metric}$\\$
	
	Similarly, we can calculate the equation of a mean curvature  flow near another mean curvature flow. Given a mean curvature flow $\{M_t\}$ defined for $t\in [0,T]$ and suppose that $M_t$ has uniformly bounded second fundamental form. In fact if $M_0$ has bounded second fundamental form, then by Chen-Yin's \cite{CY} pseduolocality theorem,  $M_t$ has uniformly bounded second fundamental form $t$ small enough. Given a function $u:M\times [0,T]\to \IR$, we define $M_{u,t}$ by 
	\begin{equation}\label{eq1.7}
		\begin{split}
			M_{u,t}=\{x+u\mathbf{n}|x\in M_t\}.
		\end{split}
	\end{equation}
	\begin{proposition}\cite{SX}\label{prop1.5}
		The graph of $M_{u,t}$ flow by MCF if and only if $u$ satisfy
		\begin{equation}\label{eq1.8}
			\partial_tu=L_{x,t}u+\mathcal{Q}_t(u),
		\end{equation}
		where $L_{x,t}u=\Delta_{M_t}u+|A_{M_t}|^2u$.
	\end{proposition}

	\begin{proposition}\cite{SX}
		There is $\delta>0$ such that if $\|u\|_{C^{0,1}}<\delta,\|v\|_{C^{0,1}}\leq \delta$, then we have 
		\begin{equation}\label{eq1.9}
			\begin{split}
				|\mathcal{Q}_t(u)|\leq C(|u|^2+|\nabla u|^2+|\nabla^2 u|(|\nabla u|+|u|)),
			\end{split}
		\end{equation}
		and 
		\begin{equation}\label{eq1.10}
			\begin{split}
				|\mathcal{Q}_t(u)-\mathcal{Q}_t(v)|&\leq C(\|u\|_{C^{0,1}}+|\nabla v|_{C^{0,1}})(|u-v|+|\nabla(u-v)|+|\nabla^2 (u-v)|)\\
				&+C(|\nabla (u-v)|+|u-v|)|\nabla^2 u|.
			\end{split}
		\end{equation}
	\end{proposition}
	
	\section{Heat kernel estimates}
	In this section, we collect some results about heat kernel of Sch\"odinger operator. 
	\subsection{Harmonic radius}$\\$
	We first introduce the following parabolic Sobolev and H\"older norms
	\begin{definition}
		For any $\Omega\subset \IR^n\times \IR$, assume that
		\begin{equation}
			r=\min\{r':\Omega\subset B_{r'}(q)\times [t-r'^2,t]\}<\infty.
		\end{equation}
		Then, we set
		\begin{align}
			&\|u\|_{W_p^{2m,m}(\Omega)} 
			:=\sum_{|J|+2k\leq  2m} r^{|J|+2k-\frac{n+2}{p}}\|D^J \partial_t^ku\|_{L^p}, \\
			&\|u\|_{C^{m+\alpha, \frac{m+\alpha}{2}}(\Omega)}
			:=\sum_{|J|+2k\leq m} \left(r^{|J|+2k}\|D^J\partial_t^ku\|_{C^0}+r^{|J|+2k+\alpha}[D^J \partial_t^ku]_{\alpha,\frac{\alpha}{2}} \right),  
		\end{align}
		where $J$ is a multi-index. 
	\end{definition}
	Throughout this paper, $\Omega$ denotes a parabolic cylinder of the form $Q_r(x, t) = B_r(x) \times [t-r^2, t]$. We have the following Sobolev embedding properties.
	\begin{lemma}\cite{Kry24}\label{sob}
		For any $m \in \IN$, $n+2<p<\infty$ and $\alpha=1-(n+2)/p$, we have 
		\begin{align}
			W_p^{2m,m}(\Omega) \hookrightarrow C^{2m-1+\alpha,\frac{2m-1+\alpha}{2}}(\Omega).
		\end{align}
		In particular, there is a constant $C=C(n,m,p)>0$ such that 
		\begin{align}
			\|u\|_{C^{2m-1+\alpha,\frac{2m-1+\alpha}{2}}}
			\leq C  \|u\|_{W_p^{2m,m}(\Omega)}.
		\end{align}
	\end{lemma}
	\begin{definition}($W^{k,p}$ harmonic radius) The $W^{k,p}$ harmonic radius at $x$ and $n<p<\infty$, is the supremum of all $R>0$ such that  there exists a coordinate chart $\phi:B_R(x)\to \IR^n$ satisfying 
		\begin{itemize}
			\item $\Delta \phi^j=0$ on $B_{R}(x)$  and $\phi^j(x)=0$ for each $j$,
			\item $2^{-1}(\delta_{ij})\leq (g_{ij})\leq 2 (\delta_{ij})$,  in $B_{R}(x)$ as symmetric bilinear forms,
			\item $\sum_{1\leq |J|\leq k}R^{|J|-\frac{n}{p}}\|\partial^J g_{ij}\|_{L^p(B_R(x))}\leq 1$.
		\end{itemize}
		We denote this radius by $r_{harm,W^{k,p}}(x)$.
	\end{definition}
	\begin{lemma}\cite[Lemma 7.1]{CY}\label{lm2.1}
		Suppose that $M^{n}\subset \IR^{n+1}$ is a smooth complete submanifold such that
		\begin{equation}
			\begin{split}
				\sup\limits_{x\in M}|A|(x)\leq \kappa.
			\end{split}
		\end{equation}
		Then there exists $i_0=i_0(n,\kappa)$ such that the injective radius $inj_M\geq i_0$. Moreover, 
		\begin{equation}
			\begin{split}
				\frac{1}{2}\omega_nr^2\leq  \mathrm{Vol}(B_r(x))\leq 2\omega_n r^2
			\end{split}
		\end{equation}
		for any $x\in M $ and $0<r<i_0\leq 1$, where $\omega_n$ is the volume of a unit ball and $B_r(x)$ is the intrinsic ball of $M$.
	\end{lemma}
	\begin{proposition}\label{prop1.8}\cite{And90}
		Let $B_{2r}(x)$ be a compact ball in the Riemannian manifold $(M,g)$. Assume that there are numbers $m\in \IN^*,\varepsilon,r>0,c_0,\cdots,c_m>0$ with 
		\begin{equation}
			|\nabla^j\mathrm{Rm}(y)|\leq c_j,r_{inj}(y)\geq r \text{ for all }y\in B_{2r}(x),j\in\{0,\cdots,m\}.
		\end{equation}
		Then, there exists a constant 
		\begin{equation}
			C_m=C(n,m,p,r,c_1,\cdots,c_m)>0
		\end{equation}
		such that $r_{harm,W^{m+2,p}}(x)\geq C_m$.
	\end{proposition}
	
	\begin{corollary}\label{coro3.5}
		Suppose that $M^{n}\subset \IR^{n+1}$ is a smooth complete submanifold such that
		\begin{equation}
			\begin{split}
				\sup\limits_{x\in M}|A|(x)\leq \kappa.
			\end{split}
		\end{equation}
		Then there exists $r_\kappa=r_\kappa(n,\kappa)$ such that $r_{harm,W^{2,p}}(x)\geq r_\kappa$.
	\end{corollary}
	\begin{proof}
		The Corollary follows by Lemma \ref{lm2.1} and  Proposition \ref{prop1.8}.
	\end{proof}
	While constructing a time-dependent coordinate system that remains harmonic on each time slice of an evolving manifold is generally unattainable, the intrinsic regularizing effect of the geometric flow ensures that the metric becomes progressively smoother. Consequently, explicit control over the harmonic radius is not strictly required to establish heat kernel estimates.
	\begin{definition}($W^{m,p}$ radius) \label{def3.8}Let $((M,g(t))_{t\in [0,T]}$ be an evolving manifold. The $\Lambda$-$W^{m,p}$ harmonic radius at $(x,t)$ and $n+2<p<\infty$, is the supremum of all $R>0$ such that  there exists $\Lambda>1$ and a coordinate chart $\phi:B_R(x)\times [t-R^2,t]\to \IR^n\times \IR^+$ satisfying 
		
		\begin{itemize}
			\item    The metric $(g_{ij})$ is uniformly equivalent to the Euclidean metric $(\delta_{ij})$ within the parabolic cylinder $\Omega_R(x, t)$. 
			\begin{equation}
				\Lambda
				^{-1}(\delta_{ij}) \leq (g_{ij}) \leq \Lambda(\delta_{ij}).
			\end{equation}
			
			\item Scale-Invariant Sobolev Regularity: The metric satisfies a bounded $L^p$ norm for its derivatives up to order $m$
			\begin{equation}\label{eqsob}
				\sum_{1 \leq |J| \leq m} R^{|J| - \frac{n+2}{p}} \|\partial^J g_{ij}\|_{L^p(\Omega_R(x, t))} \leq \Lambda.
			\end{equation}
		\end{itemize}
		where $\Omega_R(x,t)=B_{R}(x)\times [t-R^2,t]$.    We denote this radius by $r_{\Lambda,W^{m,p}}(x)$.
	\end{definition}
	\begin{corollary}
		Let $\{M_t\}_{t\in [0,T]}$ be a mean curvature flow.  Assume that
		\begin{equation}
			\sup_{t\in [0,T]}\sup_{x\in M_t}|A|(x,t)\leq \kappa
		\end{equation}
		Then, there exists a constant $\Lambda$ and
		\begin{equation}
			C_m=C_m(n,m,\kappa)>0
		\end{equation}
		such that $r_{\Lambda,W^{m+2,p}}((x,t))\geq C_m\sqrt{t}$ for any $x\in M,0<t\leq \min \{1,T\}$.
	\end{corollary}
	
	\begin{proof}
		By Proposition \ref{propeh}, we have
		\begin{equation}
			\sup\limits_{x\in M_t}|\nabla^m A|^2\leq C_m\kappa^2(1+t^{-1})^m.
		\end{equation}
		For any $x_0\in M_t$, by Proposition \ref{prop1.8}, we can choose a $W^{m+2,p}$-harmonic coordinate with respect $g(t)$ and the radius is bound below by $C_m\sqrt{t}$. They we need to show that $g^{ij}(s)$ in such a local coordinates satisfies the properties in Definition \ref{def3.8}. For any vector $v \in T_x M$, the change in its length satisfies
		\begin{equation}
			\left| \partial_s \ln(g_{ij}v^i v^j) \right| = \left| \frac{-2Hh_{ij}v^i v^j}{g_{ij}v^i v^j} \right| \leq 2|H| \cdot |A| \leq C
		\end{equation}
		Integrating from $s$ to $t$, we obtain
		\begin{equation}
			e^{-C(t-s)} g_{ij}(t) \leq g_{ij}(s) \leq e^{C(t-s)} g_{ij}(t).
		\end{equation}
		In harmonic coordinates, the initial metric $g(t)$ satisfies
		\begin{equation}
			C^{-1}\delta_{ij} \leq g_{ij} \leq C\delta_{ij}
		\end{equation}
		Therefore, the metric $g(s)$ satisfies
		\begin{equation}
			C^{-1}\delta_{ij} \leq g_{ij}(s) \leq C\delta_{ij}. 
		\end{equation}
		Similarly, we can prove \eqref{eqsob} for $s\in [t/2,t]$. For $s\in [t/2,t]$,
		\begin{equation}
			\sup\limits_{x\in M_s}|\nabla^m A|^2\leq C_m\kappa^2t^{-m}.
		\end{equation}
		The evolution of the Christoffel symbols $\Gamma_{ij}^k$ is given by
		\begin{equation}
			\frac{\partial}{\partial \tau} \Gamma_{ij}^k = g^{kl} \left( \nabla_i (H h_{jl}) + \nabla_j (H h_{il}) - \nabla_l (H h_{ij}) \right)
		\end{equation}
		Briefly, this is expressed as 
		\begin{equation}
			\frac{\partial \Gamma}{\partial \tau} = \nabla(A \ast A).
		\end{equation}
		Consequently, the evolution of the $k$-th order derivatives of the metric $\nabla^k g$ (with respect to a fixed background connection) is controlled by $A$ and its derivatives
		up to order $k$
		\begin{equation}
			\left| \frac{\partial}{\partial \tau} \nabla^k g \right| \le C(n, k) \sum_{j=0}^k |\nabla^j A| \cdot |\nabla^{k-j} A|
		\end{equation}
		We consider the evolution of the $W^{k,p}$ norm
		\begin{equation}
			\|g(s)\|_{W^{k,p}(g_s)} \le \|g(t)\|_{W^{k,p}(g_t)} + \int_s^t \left\| \frac{\partial}{\partial \tau}  g(\tau) \right\|_{W^{k,p}} d\tau
		\end{equation}
		Since all $|\nabla^j A|$ are bounded for $j \le k$ on the time interval $[s,t]$, the integral remains finite. By Gronwall's Inequality, we have
		\begin{equation}
			\|g(s)\|_{W^{k,p}} \le e^{C(t-s)} \|g(t)\|_{W^{k,p}}
		\end{equation}
		Thus, the Sobolev regularity $W^{k,p}$ is preserved at time $s$.
	\end{proof}
	To get the high order estimates of the heat kernel. We need the following Clader\'on-Zygmund inequality:
	\begin{lemma}\label{Lp-estimate}\cite{GT}
		Let $u\in W^{2,p}_{loc}\cap L^\infty(\Omega),1<p<\infty$ be a strong solution of the equation 
		\begin{equation}
			(\partial_t-Lu)=f.
		\end{equation}
		Assume that the linear operator $L$ has the form
		\begin{equation}
			Lu=a^{ij}\partial^2_{ij}u+b^i\partial_i u+cu.
		\end{equation}
		The coefficients of $L$ satisfy, for positive constants $\Lambda$,
		\begin{equation}
			\begin{split}
				& a^{ij}\in C^0(\Omega),\; b^i \in L^\infty(\Omega),c,f\in L^p(\Omega);\\
				&\Lambda^{-1}|\xi|^2\leq a^{ij}\xi_i\xi_j\leq \Lambda |\xi|^2,\quad \forall \xi \in \IR^n;\\
				&|a^{ij}|,r|b^i|,r^2\|c\|_{L^p}\leq \Lambda,
			\end{split}
		\end{equation}
		where $i,j=1,\cdots,n$. Then, there exists a constant $C=C(n,p,\Lambda)>0$ so that
		\begin{equation}
			\|u\|_{W^{2,p}(\Omega')}\leq C (\|u\|_{L^\infty(\Omega)}+\|f\|_{L^p(\Omega)}).
		\end{equation}
	\end{lemma}
	\begin{remark}
		In \cite{GP}, they showed that if $M^n$ has bounded Ricci curvature and a
		positive injectivity radius then one have $L^p$-Calderon-Zygmund inequality with constants depending only on $n,p ,\|Ric\|_{L^\infty}$ and the injectivity radius one can generalize their methods to parabolic equations. 
	\end{remark}
	
	\subsection{Heat kernel over a manifold}
	$\\$
	
	In the rest of the section, we assume that $T$ small enough and $0<T<r_\kappa^2$.
	
	\begin{proposition}\label{prop2.3}
		Let $M^n \subset \mathbb{R}^{n+1}$ be a hypersurface with the second fundamental form satisfying $\sup_M |A| \leq \kappa$. Let $G(x,t;y,s)$ denote the heat kernel associated with the heat operator $\partial_t - \Delta_M$. For any $x, y \in M$ and $0 < s < t \leq T$, the following estimate holds:
		\begin{equation}\label{eq3.16}
			\begin{split}
				0<G(x,t;y,s)&\leq\frac{C}{(t-s)^{\frac{n}{2}}}\exp\bigg(-\frac{d^2(x,y)}{4D(t-s)}\bigg),\\
				|\nabla_xG(x,t;y,s)|&\leq\frac{C}{\sqrt{t-s}(t-s)^{\frac{n}{2}}}\exp\bigg(-\frac{d^2(x,y)}{4D(t-s)}\bigg),\\
				|\nabla^2_xG(x,t;y,s)|&\leq \frac{C}{(t-s)(t-s)^{\frac{n}{2}}}\exp\bigg(-\frac{d^2(x,y)}{4D(t-s)}\bigg).\\
			\end{split}
		\end{equation}
	\end{proposition}
	\begin{proof}
		By  Gauss-Codazzi equation, we have
		\begin{equation}
			\begin{split}
				R_{ijkl}=-h_{ik}h_{jl}+h_{il}h_{jk},
			\end{split}
		\end{equation}
		where $\{h_{ij}\}$ is the second fundamental from in local coordinate. So we have $|Rm|\leq c(n)\kappa^2$. By Li-Yau's \cite{LY} theory, the heat kernel $G(x,t;y,s)$ of $\partial_t-\Delta$ satisfies
		\begin{equation}
			\begin{split}
				G(x,t;y,s)&\leq \frac{C}{|\sqrt{|B(x,\sqrt{t-s})|}\sqrt{|B(y,\sqrt{t-s})|}}e^{-\frac{d^2(x,y)}{4D(t-s)}}\\
				&\leq \frac{C}{(t-s)^{\frac{n}{2}}}e^{-\frac{d^2(x,y)}{4D(t-s)}}.\\
			\end{split}
		\end{equation}
		for $0<t-s<r_\kappa^2 $, and in the last inequality we use Lemma \ref{lm2.1}. Here $C,D$ depend only on $n,\kappa$.
		
		Since, $G(x,t;y,s)=G(x,t-s;y,0)$, we assume $s=0$.	Fix $x_0,y_0\in M,0<t<r_\kappa^2$. In $W^{2,p}$ harmonic coordinate, we have
		\begin{equation}
			\partial_\tau G(x,\tau;y_0,0)-g^{ij}(x)\partial_{ij}G(x,\tau;y_0,0)=0, \quad (x,\tau)\in \Omega(x_0,t)
		\end{equation}
		where
		\begin{equation}
			\Omega:=B_{\sqrt{t}}(x_0)\times [\frac{t}{2},t],\quad \Omega':=B_{\sqrt{t}/2}(x_0)\times [\frac{3t}{4},t].
		\end{equation}
		Noting that $g^{ij}\in W^{2,p}$, by Lemma \ref{Lp-estimate},
		\begin{equation}
			\|G\|_{W^{4,p}(\Omega')}\leq C\|G\|_{L^\infty(\Omega)},
		\end{equation} 
		Finally, by Sobolev embedding theorem (Lemma \ref{sob}), we get the last  two inequalities in \eqref{eq3.16}.
		For more details, see \cite{HW26}. 
	\end{proof}
	\begin{corollary}
		Suppose that $(y,s)\in \{M\times [0,t]\}\setminus \Omega(x,t)$, where $\Omega(x,t)=B_{\sqrt{t}}(x)\times [\frac{t}{2},t]$. We have 
		\begin{equation}
			\begin{split}
				|G(x,t;y,s)|+\sqrt{t}|\nabla_x G(x,t;y,s)|+t |\nabla^2_x G(x,t;y,s)|\leq Ct^{-\frac{n}{2}}\exp\bigg(-\frac{d^2(x,y)}{4Dt}\bigg).
			\end{split}
		\end{equation}
	\end{corollary}
	\begin{proof}
		We only prove $G(x,t;y,s)\leq Ct^{-\frac{n}{2}}\exp\bigg(-\frac{d^2(x,y)}{4Dt}\bigg)$.	If $s\leq \frac{t}{2}$,
		\begin{equation}
			\begin{split}
				G(x,t;y,s)&\leq C(t-s)^{-\frac{n}{2}}e^{-\frac{d^2(x,y)}{-4D(t-s)}}\\
				&\leq C(t/2)^{-\frac{n}{2}}e^{-\frac{d^2(x,y)}{4D(t-s)}}.
			\end{split}
		\end{equation}
		If $\frac{t}{2}\leq s\leq t$, then $d^2(x,y)\geq t$. We have
		\begin{equation}
			\begin{split}
				G(x,t;y,s)&\leq C(t-s)^{-\frac{n}{2}}e^{-\frac{d^2(x,y)}{4D(t-s)}}\\
				&\leq C(t-s)^{-\frac{n}{2}}e^{-\frac{t+d^2(x,y)}{8D(t-s)}}\\
				&\leq C\bigg(\frac{t}{t-s}\bigg)^{-\frac{n}{2}}e^{-\frac{t}{8D(t-s)}}t^{-\frac{n}{2}}e^{-\frac{d^2(x,y)}{8D(t-s)}}\\
				&\leq Ct^{-\frac{n}{2}}e^{-\frac{d^2(x,y)}{8Dt}}.
			\end{split}
		\end{equation}
		In the last inequality, we use $0<t-s\leq t/2$.
	\end{proof}

	\begin{proposition}\label{prop2.2}
		Let $M^n \subset \mathbb{R}^{n+1}$ be a hypersurface with the second fundamental form satisfying $\sup_M |A| \leq \kappa$. Let $G(x,t;y,s)$ denote the heat kernel associated with the heat operator $\partial_t - \Delta_M-|A|^2$. For any $x, y \in M$ and $0 < s < t \leq T$, the following estimate holds:
		\begin{equation}\label{eq2.3}
			\begin{split}
				K(x,t;y,s)&\leq C(t-s)^{-\frac{n}{2}}e^{-\frac{d^2(x,y)}{4D(t-s)}}\\
				|\nabla K|(x,t;y,s)&\leq C(t-s)^{-\frac{n+1}{2}}e^{-\frac{d^2(x,y)}{4D(t-s)}}
			\end{split}
		\end{equation}
		for any $0<t-s\leq r_\kappa^2$ and $C,D$ depends only on $n,\kappa$.
	\end{proposition}
	\begin{proof}
		\begin{equation}
			\begin{split}
				(\partial_t-	\Delta)(G-e^{-\kappa^2t}K)=e^{-\kappa^2t}(\kappa^2-|A|^2)K\geq 0 .
			\end{split}
		\end{equation}
		By maximum principle, we know that 
		\begin{equation}
			\begin{split}
				K(x,t;y,s)\leq e^{\kappa^2(t-s)}G(x,t;y,s)\leq \frac{C}{(t-s)^{\frac{n}{2}}}e^{-\frac{d^2(x,y)}{4Dt}}.
			\end{split}
		\end{equation}
		In $W^{2,p}$ harmonic coordinate, we have
		\begin{equation}
			\partial_tK-g^{ij}(x)\partial_{ij}K+|A|^2K=0
		\end{equation}
		where $g^{ij}\in W^{2,p}$. So, by Calder\'on-Zygmund inequality and Sobolev embedding inequality, we have
		\begin{equation}
			\begin{split}
				|\nabla K|(x,t;y,s)\leq C(t-s)^{-\frac{n+1}{2}}e^{-\frac{d^2(x,y)}{4D(t-s)}}.
			\end{split}
		\end{equation}
	\end{proof}
	
	\begin{corollary}
		Suppose that $(y,s)\in \{M\times [0,t]\}\setminus \Omega(x,t)$, where $\Omega(x,t)=B_{\sqrt{t}}(x)\times [\frac{t}{2},t]$. We have 
		\begin{equation}
			\begin{split}
				|K(x,t;y,s)|+\sqrt{t}|\nabla_x K(x,t;y,s)|\leq Ct^{-\frac{n}{2}}\exp\bigg(-\frac{d^2(x,y)}{4Dt}\bigg).
			\end{split}
		\end{equation}
	\end{corollary}
	\subsection{Heat kernel over evolving metric}$\\$
	
	In  Section \ref{sec4}, we need estimates of the heat  kernel for an evolving metric. Let $\{M_t\}_{t\in [0,T]}$ be a mean curvature flow such that $\sup\limits_{x\in M_0}|A|(x)\leq \kappa$. By Chen-Yin's pseudolocality (cf.\cite[Corollary 1.5]{CY}), we have 
	\begin{equation}\label{eq2.14}
		\sup\limits_{(x,t)\in M\times [0,T]}|A|_{g_t}(x,t)\leq C,
	\end{equation}
	here $T=T(n,\kappa)$. Moreover, by Ecker-Husiken's interior estimates, we have 
	\begin{proposition}\cite{EH}\label{propeh}
		Let $\{M_t\}_{t\in [0,T]}$ be a mean curvature flow such that $\sup\limits_{x\in M_0}|A|(x)\leq \kappa$. Then there exist $c_m>0$ such that
		\begin{equation}
			\sup\limits_{x\in M_t}|\nabla^m A|^2\leq C_m\kappa^2(1+t^{-1})^m.
		\end{equation}
	\end{proposition}

	Let $L_{x,t}$ be an operator on $\{M_t\}_{t\in [0,T]}\subset \IR^{n+1}$,  defined as 
	\begin{equation}
		L_{x,t}u:=\frac{\partial }{\partial t}u-\Delta_{x,t}u-|A|^2(x,t)u
	\end{equation}
	where $A$ is the second fundamental form of $M_t$ for $t\in [0,T]$. Define $g(t)=g_t$ to be the pull-back metric from $M_t$, we have
	\begin{equation}
		\frac{\partial }{\partial t}g(t)=-2HA
	\end{equation}
	where $H$ is the mean curvature of $M_t\subset \IR^{n+1}$ and $A$ is the second fundamental form. For the existence and uniqueness of the heat kernel on evolving metrics, we recommend  the book \cite{Cho} for details.
	
	\begin{theorem}\cite[Theorem 24.40]{Cho}\label{cho}
		Let $M^n$ be a complete manifold and let $g(t),t\in [0,T]$, be a smooth family Riemannian metrics on $M$. If $Q$ is uniformly bounded, then there exists a unique $C^\infty$ minimal  positive fundamental solution $G_Q(x,\tau;y,v)$ for the heat-type  operator  $L_{x,\tau}=\frac{\partial}{\partial \tau}-\Delta_{x,t}+Q$. 
	\end{theorem}
	Letting $Q\equiv 0$ in Theorem \ref{cho}, we have
	\begin{proposition}\label{prop3.8}
		Let $\tilde{G}$ be the heat kernel of operator $\partial_t-\Delta_{x,t}$.	There are constants $C,D>0$ depending on $m$ such that
		\begin{equation}
			\begin{split}
				\tilde{G}(x,t;y,s)&\leq \frac{C}{(t-s)^{\frac{n}{2}}}\exp\bigg(-\frac{d_{g(t)}^2(x,y)}{4D(t-s)}\bigg),\\
				\sum_{2i+j\leq m}|\partial^i_t\nabla^{j}_x\tilde{G}(x,t;y,s)|&\leq\frac{C}{(t-s)^{\frac{n+m}{2}}}\exp\bigg(-\frac{d_{g(t)}^2(x,y)}{4D(t-s)}\bigg),
			\end{split}
		\end{equation}
		for any $0<s<t<1$.
	\end{proposition}
	\begin{proof}
		We can write the heat kernel equation in the local coordinate and using Calder\'on-Zygmund inequality and Sobolev embedding theorem.
	\end{proof}
	Letting $Q= |A|^2$ in Theorem \ref{cho}, we have
	\begin{proposition}\label{prop2.4}
		
		Let $\tilde{K}(x,t;y,s)$ be the heat kernel for the operator $L_{x,t}$. 	There are constants $C,D>0$ such that
		\begin{equation}
			\begin{split}
				\tilde{K}(x,t;y,s)&\leq \frac{C}{(t-s)^{\frac{n}{2}}}\exp\bigg(-\frac{d_{g(t)}^2(x,y)}{4D(t-s)}\bigg),\\
				\sum_{2i+j\leq k}|\partial^i_t\nabla^{j}_x\tilde{K}(x,t;y,s)|&\leq\frac{C}{(t-s)^{\frac{n+k}{2}}}\exp\bigg(-\frac{d_{g(t)}^2(x,y)}{4D(t-s)}\bigg),
			\end{split}
		\end{equation}
		for any $j\in \IN^+$ and $0<s<t<1$.
	\end{proposition}
	
	\begin{corollary}
		Suppose that $(y,s)\in \{M\times [0,t]\}\setminus \Omega(x,t)$, we have 
		\begin{equation*}
			\begin{split}
				|\tilde{G}(x,t;y,s)|+\sqrt{t}|\nabla_x \tilde{G}(x,t;y,s)|+t |\nabla^2_x \tilde{G}(x,t;y,s)|\leq Ct^{-\frac{n}{2}}\exp\bigg(-\frac{d_{g(t)}^2(x,y)}{4Dt}\bigg).
			\end{split}
		\end{equation*}
		and
		\begin{equation*}
			\begin{split}
				|\tilde{K}(x,t;y,s)|+\sqrt{t}|\nabla_x \tilde{K}(x,t;y,s)|+t |\nabla^2_x \tilde{K}(x,t;y,s)|\leq Ct^{-\frac{n}{2}}\exp\bigg(-\frac{d_{g(t)}^2(x,y)}{4Dt}\bigg).
			\end{split}
		\end{equation*}
	\end{corollary}

	\section{Existence and uniqueness}\label{sec3}
	\begin{definition}\label{def3.1}
		For every $0<T<\infty$, we define the function spaces
		\begin{equation}
			\begin{split}
				X_T=\bigg\{f|\|f\|_{X^T}=&\lim\limits_{0<t<T}\|f\|_{L^\infty(M)}+\lim\limits_{0<t<T}\|\nabla f\|_{L^\infty(M)}\\
				&+\lim\limits_{x\in M}\sup\limits_{0<r^2<T}r^{\frac{2}{n+4}}\|\nabla^2 f\|_{L^{n+4}(\Omega(x,r))}\bigg\}
			\end{split}
		\end{equation}
		and
		\begin{equation}
			\begin{split}
				Y_T=\bigg\{g|\|g\|_{Y_T}=\lim\limits_{x\in M}\sup\limits_{0<r^2<T}r^{\frac{2}{n+4}}\|g\|_{L^{n+4}(\Omega(x,r))}\bigg\}
			\end{split}
		\end{equation}
		where 
		\begin{equation}
			\begin{split}
				P(x,r):=B(x,r)\times (0,r^2),\mathrm{and}\quad \Omega(x,r):=B(x,r)\times (\frac{r^2}{2},r^2).
			\end{split}
		\end{equation}
	\end{definition}
	In this section, let $H$ be the fundamental solution for $\partial_t-\Delta_M$ and $K$ be the fundamental solution for $\partial_t-L$ in \eqref{eq1.4}. Consider equation \eqref{eq1.4}. We have 
	\begin{equation}
		u(x,t)=\int_{M\times [0,t]}(K(x,t;y,s)(-H_0+\mathcal{Q}(u))dyds.
	\end{equation}
	\begin{proposition}\label{prop3.1}
		There exists a constant $C_1=C_1(n,\kappa)>0$ such that 
		\begin{equation}
			\bigg\|\int_M\int_0^tK(x,t;y,s)Q(y,s)dyds\bigg\|_{X_T}\leq C\|Q\|_{Y_T}.
		\end{equation}
	\end{proposition}
	\begin{proof}
		We define 
		\begin{equation}
			g(x,t):=\int_M\int_0^tK(x,t;y,s)Q(y,s)dyds
		\end{equation}
		
		\textbf{ Step 1. Estimate $\|g\|_{L^\infty}$:}
		
		\begin{equation}
			\begin{split}
				|g(x,t)|&=\int_{\Omega(x,t)}K(x,t;y,s)Q(y,s)dyds\\
				&+\int_{M\times [0,t]\setminus\Omega(x,t)}K(x,t;y,s)Q(y,s)dyds\\
				&=\uppercase\expandafter{\romannumeral1}+\uppercase\expandafter{\romannumeral2}.
			\end{split}
		\end{equation}
		By H\"older inequality,
		\begin{equation}\label{eq3.8}
			\begin{split}
				\uppercase\expandafter{\romannumeral1}&\leq \|Q\|_{L^{n+4}(\Omega(x,t))}\|K\|_{L^{\frac{n+4}{n+3}}(\Omega(x,t))} .
			\end{split}
		\end{equation}
		Using inequality \eqref{eq2.3}, we have
		\begin{equation}
			\begin{split}
				&\int_{t/2}^t\int_{B(x,\sqrt{t})}|K(x,t;y,s)|^{\frac{n+4}{n+3}}dyds \\
				&\quad \leq C\int_0^{t/2}\int_{B(x,\sqrt{t})}\tau^{-\frac{n(n+4)}{2(n+3)}}e^{-\frac{\frac{n+4}{n+3}d^2(x,y)}{4D\tau}}dyd\tau \\
				&\quad = C\int_0^{t/2}\tau^{-\frac{n}{2(n+3)}}\int_{B(x,\sqrt{t})}\tau^{-\frac{n}{2}}e^{-\frac{\frac{n+4}{n+3}d^2(x,y)}{4D\tau}}dyd\tau \\
				&\quad \leq C\int_0^{t/2}\tau^{-\frac{n}{2(n+3)}}d\tau \\
				&\quad \leq Ct^{\frac{n+6}{2(n+3)}}.
			\end{split}
		\end{equation}
		i.e,
		\begin{equation}\label{eq3.10}
			\begin{split}
				\|K(x,t;\cdot,\cdot)\|_{L^{\frac{n+4}{n+3}}(\Omega(x,t))}\leq Ct^{\frac{n+6}{2(n+4)}}.
			\end{split}
		\end{equation}
		Similarly, by Proposition \ref{prop2.3}, we obtain
		\begin{equation}
			\begin{split}
				\|\nabla K(x,t;\cdot,\cdot)\|_{L^\frac{n+4}{n+3}}\leq Ct^{-\frac{n+3}{2(n+4)}}Ct^{\frac{n+6}{2(n+4)}}= Ct^{\frac{3}{2(n+4)}}.
			\end{split}
		\end{equation}
		Combining \eqref{eq3.8} and \eqref{eq3.10}, 
		\begin{equation}\label{eq3.12}
			\begin{split}
				\uppercase\expandafter{\romannumeral1}&\leq  Ct^{\frac{n+6}{2(n+4)}}\|Q\|_{L^{n+4}}\leq C\|Q\|_{Y_T}.
			\end{split}
		\end{equation}
		
		For any $(y,s)\in M\setminus \Omega(x,t)$. By the heat kernel estimates \eqref{eq2.3}, we have
		\begin{equation}\label{eq3.13}
			\begin{split}
				& \int_{(B(p,\sqrt{t})\times [0,t])\setminus \Omega(x,t)}|K(x,t;y,s)Q(y,s)|dyds\\
				&\quad	\leq 	Ct^{-\frac{n}{2}}\int_0^t \int_{B(p,\sqrt{t})}e^{-\frac{d^2(x,y)}{4Dt}}|Q(y,s)|dyds\\
				&\quad\leq Ct^{-\frac{n}{2}}e^{-\frac{d^2(x,p)}{4Dt}} \int_0^t\int_{B(p,\sqrt{t})}|Q(y,s)|dyds\\
				&\quad \leq Ct^{-\frac{n}{2}}e^{-\frac{d^2(x,p)}{4Dt}} \sum_{m=0}^\infty\int_{2^{-m-1}t}^{2^{-m}t}\int_{B(p,\sqrt{t})}|Q(y,s)|dyds.
			\end{split}
		\end{equation}
		By Lemma \ref{lm2.1}, we can cover $B(p,\sqrt{t})\times (2^{-m-1}t,2^{-m}t)$ by approximately $C(n,m)\sim 2^{mn}$ cylinders of the form $T_m(y_i):=B(y,2^{-m/2}\sqrt{t})\times (2^{-m-1}t,2^{-m}t)$ and we use H\"older’s inequality to estimate
		\begin{equation}
			\begin{split}
				\|Q\|_{L^1(T_m(y))}&\leq \|Q\|_{L^{n+4}(T_m(y))}\|1\|_{L^{\frac{n+4}{n+3}}(T_m(y))}\\
				&\leq C \|Q\|_{L^{n+4}(T_m(y))}(2^{-m}t)^{\frac{n(n+3)}{2(n+4)}} (2^{-(m+1)}t)^{\frac{n+3}{n+4}}\\
				&\leq C(2^{-m/2}\sqrt{t})^{n+1}(2^{-m/2}\sqrt{t})^{\frac{2}{n+4}}\|Q\|_{L^{n+4}(T_m(y))}\\
				&\leq C(2^{-m/2}\sqrt{t})^{n+1}\|Q\|_{Y_T}.
			\end{split}
		\end{equation}
		Hence,
		\begin{equation}\label{eq3.15}
			\begin{split}
				&t^{-\frac{n}{2}}\int_{2^{-m-1}t}^{2^{-m}t}\int_{B(p,\sqrt{t})}|Q(y,s)|dyds\\
				&\quad \leq \sum_{i=1}^{C(n,m)}t^{-\frac{n}{2}}\int_{2^{-m-1}t}^{2^{-m}t}\int_{T_m(y_i)}|Q(y,s)|dyds\\
				&\quad \leq C2^{\frac{mn}{2}}2^{-\frac{(n+1)m}{2}}\sqrt{t}\|Q\|_{Y_T}\\
				&\quad\leq C2^{-\frac{m}{2}}\sqrt{t}\|Q\|_{Y_T}.
			\end{split}
		\end{equation}
		Plugging it into \eqref{eq3.13},
		\begin{equation}
			\begin{split}
				& \int_{(B(p,\sqrt{t})\times [0,t])\setminus \Omega(x,t)}|K(x,t;y,s)Q(y,s)|dyds\\
				&\quad\leq C\sqrt{t}e^{-\frac{d^2(x,p)}{4Dt}}\sum_{m=0}^\infty 2^{-\frac{m}{2}}\|Q\|_{Y_T} \\
				&\quad\leq C\sqrt{t}e^{-\frac{d^2(x,p)}{4Dt}}\|Q\|_{Y_T} .
			\end{split}
		\end{equation}
		We choose $p_i$ such that $\{B(p_i,\sqrt{t})\}_i$ is a cover of $M$ and $\{B(p_i,\frac{1}{5}\sqrt{t})\}_i$ are disjoint and $p_0=x$. Then we have
		\begin{equation}\label{eq3.17}
			\begin{split}
				II\leq C\sqrt{t}\sum_i\exp\bigg(-\frac{d^2(x,p_i)}{4Dt}\bigg)\|Q\|_{Y_T}\leq C\sqrt{t}\|Q\|_{Y_T}.
			\end{split}
		\end{equation}
		
		\textbf{Step 2. Estimate $\|\nabla g\|_{L^\infty}$:}
		
		\begin{equation}
			\begin{split}
				|\nabla g(x,t)|&\leq \bigg|\int_{\Omega(x,t)}\nabla K(x,t;y,s)Q(y,s)dyds\bigg|\\
				&+\bigg|\int_{M\times [0,t]\setminus\Omega(x,t)}\nabla K(x,t;y,s)Q(y,s)dyds\bigg|\\
				&=\uppercase\expandafter{\romannumeral3}+\uppercase\expandafter{\romannumeral4}.
			\end{split}
		\end{equation}
		Similar to \eqref{eq3.12} and \eqref{eq3.17}, we have
		\begin{equation}
			\begin{split}
				\uppercase\expandafter{\romannumeral3}&\leq \|Q\|_{L^{n+4}(\Omega(x,t))}\|\nabla K\|_{L^{\frac{n+4}{n+3}}(\Omega(x,t))}\\
				& \leq C t^{\frac{3}{2(n+4)}}\|Q\|_{L^{n+4}(\Omega(x,t))}\\
				&\leq C\|Q\|_{Y_T}.
			\end{split}
		\end{equation}
		and
		\begin{equation}
			\begin{split}
				\uppercase\expandafter{\romannumeral4}\leq \sum_i\exp\bigg(-\frac{d^2(x,p_i)}{4Dt}\bigg)\|Q\|_{Y_T}\leq C\|Q\|_{Y_T}.
			\end{split}
		\end{equation}
		Combining above inequalities, we get
		\begin{equation}\label{eq3.21}
			\begin{split}
				\|g(x,t)\|_{C^{0,1}}=	\|g(x,t)\|_{L^\infty}+\|\nabla g(x,t)\|_{L^\infty}\leq C\|Q\|_{Y_T}.
			\end{split}
		\end{equation}
		
		\textbf{Step 3. Estimate $\|\nabla^2 g\|_{L^{n+4}}$:}
		
		We do not have the second derivative estimate of $K$. However, $g$ satisfies the following equation
		\begin{equation}
			\partial_tg=\Delta g+|A|^2g+Q=\Delta g+\tilde{Q}.
		\end{equation}
		Here $\tilde{Q}=|A|^2g+Q$. Since we have the second derivative estimate of $H$
		
		\begin{equation}\label{eq3.22}
			\begin{split}
				|\nabla^2 g(x,t)|&=\bigg|\int_{M\times [0,t]}\nabla^2 G(x,t;y,s)\tilde{Q}(y,s)dyds\bigg|\\
				&\leq \bigg|\int_{\Omega(x_0,r)}\nabla^2 G(x,t;y,s)\tilde{Q}(y,s)dyds\bigg|\\
				&+\bigg|\int_{M\times [0,t]\setminus\Omega(x_0,r)}\nabla^2 G(x,t;y,s)\tilde{Q}(y,s)dyds\bigg|\\
				&=\uppercase\expandafter{\romannumeral5}+\uppercase\expandafter{\romannumeral6}.
			\end{split}
		\end{equation}
		By triangle inequality,
		\begin{equation}
			\begin{split}
				\|\nabla^2 g(x,t)\|_{L^{n+4}(\Omega(x_0,r))}\leq \|\uppercase\expandafter{\romannumeral5}\|_{L^{n+4}(\Omega(x_0,r))}+\|\uppercase\expandafter{\romannumeral6}\|_{L^{n+4}(\Omega(x_0,r))}.
			\end{split}
		\end{equation}
		Using the same trick as in the estimates of $\uppercase\expandafter{\romannumeral2}$, we obtain 
		\begin{equation}
			\begin{split}
				\uppercase\expandafter{\romannumeral6}\leq Ct^{-1/2}\|\tilde{Q}\|_{Y_T}\leq Ct^{-1/2}\|Q\|_{Y_T}+Ct^{-1/2}\||A|^2g\|_{Y_T}.	
			\end{split}
		\end{equation}
		Using \eqref{eq3.21}, we have
		\begin{equation}
			\begin{split}
				\||A|^2g\|_{Y_T}&\leq C\kappa^2\|Q\|_{Y_T}.
			\end{split}
		\end{equation}
		So,
		\begin{equation}\label{eq3.25}
			\begin{split}
				&r^{\frac{2}{n+4}}\|\uppercase\expandafter{\romannumeral6}\|_{L^{n+4}(\Omega(x,r))}\\
				&\quad\leq C\|Q\|_{Y_T}r^{\frac{2}{n+4}}\|t^{-\frac12}\|_{L^{n+4}(\Omega(x,r))}\\
				&\quad\leq C\|Q\|_{Y_T}.
			\end{split}
		\end{equation}
		
		Define 
		\begin{equation}
			\tilde{g}(x,t)=\int_{\Omega(x,r)} G(x,t;y,s)\tilde{Q}(y,s)dyds.
		\end{equation}
		By direct calculation, we have
		\begin{equation}\label{eq3.27}
			\tilde{g}_t-\Delta\tilde{g}=\tilde{Q}, \quad (x,t)\in \Omega(x_0,r),
		\end{equation}
		and
		\begin{equation}
			\tilde{g}(x,r^2/2)=0  \quad x\in B_r(x_0).
		\end{equation}
		We can extend the domain of $\tilde{g}$  to $B_{2r}(x_0)\times (\tfrac14r^2,r^2)$ by defining
		\begin{equation}
			\tilde{g}(x,t)=\int_0^t\int_{M} G(x,t;y,s)\tilde{Q}\chi_{\Omega(x_0,r)}dyds, \quad (x,t)\in B_{2r}(x_0)\times (0,r^2),
		\end{equation}
		where $\chi$ is the characteristic function.
		We can choose $r$ small enough such that \eqref{eq3.27} can be written as a parabolic equation in harmonic coordinate. Hence, choosing $p=n+4$ in Lemma \ref{Lp-estimate}, we have

		\begin{equation}\label{eq3.31}
			\begin{split}
				&r^{\frac{2}{n+4}}\|\uppercase\expandafter{\romannumeral5}\|_{L^{n+4}(\Omega(x_0,r))}\\
				&\quad\leq Cr^{\frac{2}{n+4}}\|Q\chi_{\Omega(x,r)}\|_{L^{n+4}(B_{2r}(x_0)\times (0,r^2))}\\
				&\quad= Cr^{\frac{2}{n+4}}\|Q\|_{L^{n+4}(B_{2r}(x_0)\times (r^2/2,r^2))}\\
				&\quad\leq C\|Q\|_{Y_T}.
			\end{split}
		\end{equation}
		The $L^{n+4}$ norm estimate follows form \eqref{eq3.22}, \eqref{eq3.25} and \eqref{eq3.31}.
	\end{proof}
	
	\begin{proposition}\label{prop3.2}
		Let $\mathcal{Q}(u)$ be as defined in  \eqref{eq1.4}. For any $T<1$, there exists a small constant $\delta<1$ and a large constant $C_2=C_2(n,\kappa)$ such that for each $X\in M$ and $r\in (0,\sqrt{T})$, the following property holds,
		\begin{itemize}
			\item If $\|u\|_{X_T}<\delta$, then 
			\begin{equation}
				\|\mathcal{Q}(u)\|_{Y_T}\leq C_2\|u\|^2_{X_T},
			\end{equation}
			\item If $\|u_1\|_{X_T}, \|u_2\|_{X_T}<\delta$, then
			\begin{equation}
				\|\mathcal{Q}(u_1)-\mathcal{Q}(u_2)\|_{Y_T}\leq C_2(\|u_1\|_{X_T}+\|u_2\|_{X_T})\|u_1-u_2\|_{X_T}.
			\end{equation}
		\end{itemize}
	\end{proposition}
	\begin{proof}
		From Proposition \ref{prop1.4}, we know that
		\begin{equation}
			\begin{split}
				|\mathcal{Q}(u)|&\leq C|\nabla^2 u|(|\nabla u|+|u|)+C(|\nabla u|^2+|u|^2)\\
				&\leq C|\nabla^2 u|\|u\|_{C^{0,1}}+C\|u\|^2_{C^{0,1}}
			\end{split}
		\end{equation}
		and
		\begin{equation}\label{eq3.30}
			\begin{split}
				|Q(u_1)-Q(u_2)|&\leq |\nabla^2(u_1-u_2)|(|u_1|+|u_2|+|\nabla u_1|+|\nabla u_2|)\\
				&+(|\nabla(u_1-u_2)|)(|u_1-u_2|+|\nabla(u_1-u_2)|)\\
				&\leq C|\nabla^2 (u_1-u_2)|\|u_1\|_{C^{0,1}}+|\nabla^2u_1|\|u_1-u_2\|_{C^{0,1}}\\
				&+C(\|u_1\|^2_{C^{0,1}}+\|u_2\|^2_{C^{0,1}})(\|u_1-u_2\|^2_{C^{0,1}}).
			\end{split}
		\end{equation}
		Hence, we have
		\begin{equation}
			\begin{split}
				\|\mathcal{Q}(u)\|_{L^{n+4}(\Omega(x,r))}&\leq C(	\||\nabla^2 u|\|u\|_{C^{0,1}}\|_{L^{n+4}(\Omega(x,r))}+	\|u\|^2_{C^{0,1}}\|1\|_{L^{n+4}(\Omega(x,r))})\\
				&=C(\|u\|_{C^{0,1}}\|\nabla^2 u\|_{L^{n+4}(\Omega(x,r))}+\|u\|^2_{C^{0,1}}\|1\|_{L^{n+4}(\Omega(x,r))})
			\end{split}
		\end{equation}
		i.e,
		\begin{equation}
			\begin{split}
				\|\mathcal{Q}(u)\|_{Y_T}
				&=\sup\limits_{x\in M,0<r^2<T} r^{\frac{2}{n+4}}\|\mathcal{Q}(u)\|_{L^{n+4}(\Omega(x,r))}\\
				&\leq C\sup\limits_{x\in M,0<r^2<T} r^{\frac{2}{n+4}}\|u\|^2_{C^{0,1}}\|1\|_{L^{n+4}(\Omega(x,r))}\\
				&+ C\sup\limits_{x\in M,0<r^2<T} r^{\frac{2}{n+4}}\|u\|_{C^{0,1}}\|\nabla^2 u\|_{L^{n+4}(\Omega(x,r))}\\
				&\leq C\|Q\|_{X_T}^2.
			\end{split}
		\end{equation}
		By \eqref{eq3.30},	we conclude that
		\begin{equation}
			\begin{split}
				\|\mathcal{Q}(u_1)-\mathcal{Q}(u_2)\|_{Y_T}&\leq C(\|u_1\|_{X_T}+\|u_2\|_{X_T})\|u_1-u_2\|_{X_T}.
			\end{split}
		\end{equation}
	\end{proof}
	
	\begin{proposition}\label{prop3.3}
		There exists a constant $C>0$ such that 
		\begin{equation}
			\bigg\|\int_{M\times [0,t]}K(x,t;y,s)H_0dyds \bigg\|_{X_T}\leq C_3\sqrt{T}\|H_0\|_{L^\infty(M)}.
		\end{equation}
	\end{proposition}
	\begin{proof}
		Let $Q(y,s)=H_0(y)$, by Proposition \ref{prop3.1},
		\begin{equation}
			\begin{split}
				&\bigg\|\int_{M\times [0,t]}K(x,t;y,s)H_0dyds \bigg\|_{X_T}\\
				&\quad\leq C\|Q\|_{Y_T}\\
				&\quad=C\sup\limits_{x\in M}\sup\limits_{0<r<\sqrt{t}}r^{\frac{2}{n+4}}\|Q\|_{L^{n+4}(\Omega(x,r))}\\
				&\quad\leq C\|H_0\|_{L^\infty}\sup\limits_{x\in M}\sup\limits_{0<r<\sqrt{t}}r^{\frac{2}{n+4}}\|1\|_{L^{n+4}(\Omega(x,r))}\\
				&\quad\leq C\sqrt{T}\|H_0\|_{L^\infty}.
			\end{split}
		\end{equation}
	\end{proof}

	We are ready to prove the existence of the mean curvature flow.
	\begin{theorem}\label{thm3.4}
		For any $M$ with bounded second fundamental form, there exist $T>0$ and $\delta>0$ and a unique $u(\cdot,t)\in X_T$ to the equation \eqref{eq1.4}
		\begin{equation}
			u(x,t)=\int_{M\times [0,t]}K(x,t;y,s)(-H_0+\mathcal{Q}(u))dyds,
		\end{equation}
		where $u(\cdot,0)=0$ and $\|u\|_{X_T}\leq \delta$.
		
	\end{theorem}
	
	\begin{proof}
		Denote $X^\delta_T=\{u\in X_T:u(\cdot,0)=0,\|u\|_{X_T}\leq \delta\}$. Let $\mathcal{G}:X^\delta\to X^\delta_T$ be a map defined by
		\begin{equation}
			\mathcal{G}(u):=\int_{M\times [0,t]}K(x,t;y,s)(-H_0+\mathcal{Q}(u))dyds.
		\end{equation}
		We next show that $ \mathcal{G}$ is a contraction mapping on $X^\delta_T$ for some $\delta>0$ small enough. Given any $h_1,h_2\in X_T$ such that $u_i(\cdot,0)=0, \|u_i\|_{X_T}<\delta$ for $i=1,2$, by Proposition \ref{prop3.1} and  Proposition \ref{prop3.2} , we have
		\begin{equation}
			\begin{split}
				&\| \mathcal{G}(u_1)-\mathcal{G}(u_2)\|_{X_T}\\
				&\quad=\bigg\|\int_{M\times [0,t]}K(x,t;y,s)(\mathcal{Q}(u_1)-\mathcal{Q}(u_2))dyds \bigg\|_{X_T}\\
				&\quad\leq C_1\bigg\|\mathcal{Q}(u_1)-\mathcal{Q}(u_2) \bigg\|_{Y_T}\\
				&\quad\leq C_1C_2 (\|u_1\|_{X_T}+\|u_2\|_{X_T})\|u_1-u_2\|_{X_T}\\
				&\quad\leq 2C_1C_2\delta \|u_1-u_2\|_{X_T}.
			\end{split}
		\end{equation} 
		On the other hand, by Proposition \ref{prop3.3}
		\begin{equation}
			\begin{split}
				\bigg\|\int_{M\times [0,t]}K(x,t;y,s)H_0dyds\bigg\|_{X_T}\leq C\sqrt{T}\|H_0\|_{X_T}
			\end{split}
		\end{equation}
		and by Proposition \ref{prop3.1} and Proposition \ref{prop3.2}
		\begin{equation}
			\begin{split}
				\| \mathcal{G}(u)\|_{X_T}&=\bigg\|\int_{M\times [0,t]}K(x,t;y,s)(\mathcal{Q}(u)+H_0)dyds \bigg\|_{X_T}\\
				&\leq \bigg\|\int_{M\times [0,t]}K(x,t;y,s)\mathcal{Q}(u)dyds \bigg\|_{X_T}\\
				&+\bigg\|\int_{M\times [0,t]}K(x,t;y,s)H_0dyds \bigg\|_{X_T}\\
				&\leq C_1C_2 \|u\|^2_{X_T}+C_3\sqrt{T}\|H_0\|_{L^\infty}.\\
			\end{split}
		\end{equation} 
		After taking $\delta=\frac{1}{4C_1C_2}$ and $\sqrt{T}=\mathrm{min}\{(8C_1C_2C_3\|H_0\|_{L^\infty})^{-1},i_0/2\}$, we have  
		\begin{equation}
			\begin{split}
				\| \mathcal{G}(u_1)-\mathcal{G}(u_2)\|_{X_T}\leq \frac{1}{2} \|u_1-u_2\|_{X_T}
			\end{split}
		\end{equation} 
		and 
		\begin{equation}
			\begin{split}
				\| \mathcal{G}(u)\|_{X_T}< \delta,\\
			\end{split}
		\end{equation} 
		which implies that $\mathcal{G}$ is a contraction map from  $X^\delta_T$ to itself. Finally, by the Banach fixed point theorem, there exists a unique solution $u\in X_T$ with $\|u\|_{X_T}\leq \delta$. 
	\end{proof}
	\begin{proof}[Proof of Theorem \ref{thm0.1}:] Theorem \ref{thm0.1} follows from Proposition \ref{prop1.3} and Theorem \ref{thm3.4}. 
	\end{proof}
	\section{Continuous dependence}\label{sec4}

	\begin{definition}\label{def4.1}
		Let $\{M_t\}_{t\in [0,T]}$ be a mean curvature flow. We define the function spaces
		\begin{equation}
			\begin{split}
				X_T=\bigg\{f|\|f\|_{X^T}=&\lim\limits_{0<t<T}\|f\|_{L^\infty(M_t)}+\lim\limits_{0<t<T}\|\nabla f\|_{L^\infty(M_t)}\\
				&+\lim\limits_{x\in M}\sup\limits_{0<r^2<T}r^{\frac{2}{n+4}}\|\nabla^2 f\|_{L^{n+4}(\Omega(x,r))}\bigg\}
			\end{split}
		\end{equation}
		and
		\begin{equation}
			\begin{split}
				Y_T=\bigg\{f|\|f\|_{Y_T}=\lim\limits_{x\in M}\sup\limits_{0<r^2<T}r^{\frac{2}{n+4}}\|f\|_{L^{n+4}(\Omega(x,r))}\bigg\}
			\end{split}
		\end{equation}
		where 
		\begin{equation}
			\begin{split}
				P(x,r):=\bigcup_{0<t<r^2}B_{g(t)}(x,r)\times \{t\},\mathrm{and}\quad \Omega(x,r):= \bigcup_{\frac{r^2}{2}<t<r^2}B_{g(t)}(x,r)\times  \{t\},
			\end{split}
		\end{equation}
		and
		\begin{equation}
			\|f\|^{n+4}_{L^{n+4}(\Omega(x,r))}=\int^{r^2}_{\frac{r^2}{2}}\int_{B_{g(t)}(x,r)}|f|^{n+4}d\mu_t.
		\end{equation}
	\end{definition}
	\begin{remark}
		Under the bounded curvature condition \eqref{eq2.14}, we know that  there exists a $C_0>0$ such that  
		\begin{equation}
			\frac{1}{C_0}g_t\leq g_0\leq C_0g_t.
		\end{equation}
		So, the $L^{n+4}$ norm defined in Definition \ref{def3.1} is equivalent to the $L^{n+4}$ norm defined in Definition \ref{def4.1}
		.
	\end{remark}
	\begin{proposition}\label{prop4.1}
		Let $\{M_t\}_{t\in [0,T]}$ be a mean curvature flow.  Assume that $M_0$ has bounded second fundamental form. Suppose that $T$ is small enough. There exists $C_4=C_4(n,\kappa,T)$ such that for any $Q\in Y_T$, we have
		\begin{equation}
			\bigg\|\int_M\int_0^t\tilde{K}(x,t;y,s)Q(y,s)dyds\bigg\|_{X_T}\leq C_4\|Q\|_{Y_T},
		\end{equation}
		where $\tilde{K}$ is the heat kernel of the operator $\tilde{\square}$.
	\end{proposition}
	\begin{proof}
		
		See Proposition \ref{prop3.1} for detail.
	\end{proof}
	
	\begin{proposition}\label{prop4.2}
		Let $\{M_t\}_{t\in [0,T]}$ be a mean curvature flow.  Assume that $M_0$ has bounded second fundamental form. Suppose that $T$ is small enough. There exists $C_5=C_5(n,\kappa,T)$ and $\delta=\delta(n,\kappa,T)$ such that for any
		$u,u_1,u_2\in X_T^\delta:=\{u\in X_T|\|u\|_{X^T}<\delta\}$, we have the estimates
		\begin{equation}
			\|\mathcal{Q}_t(u)\|_{Y_T}\leq C_5\|u\|^2_{X_T}
		\end{equation}
		and 
		\begin{equation}
			\|\mathcal{Q}_t(u_1)-\mathcal{Q}_t(u_2)\|_{Y_T}\leq C_5(\|u_1\|_{X_T}+\|u_2\|_{X_T})\|u_1-u_2\|_{X_T},
		\end{equation}
		where $\mathcal{Q}_t$ is defined in Proposition \ref{prop1.5}.
		
	\end{proposition}
	
	\begin{proof}
		See Proposition \ref{prop3.2} for detail.
	\end{proof}
	\begin{proposition}\label{prop4.3}
		Let $\{M_t\}_{t\in [0,T]}$ be a mean curvature flow.  Assume that $M_0$ has bounded second fundamental form. Suppose that $T$ is small enough. There exists $C_6=C_6(n,\kappa,T)$  such that for any $f_0$ with $\|f_0\|_{C^{0,1}}<\infty$ we have
		\begin{equation}
			\bigg\|\int_{M_t}\tilde{K}(x,t;y,0)f_0dy\bigg\|_{X_T}\leq C_6\|f_0\|_{C^{0,1}}.
		\end{equation}
		
	\end{proposition}
	
	\begin{proof}
		Define
		\begin{equation}
			f(x,t)=\int_{M}\tilde{K}(x,t;y,0)f_0dy.
		\end{equation}
		By Proposition \ref{prop2.4}, 
		\begin{equation}\label{eq4.5}
			\begin{split}
				|f(x,t)|\leq \|f_0\|_{L^\infty}\int_{M}\tilde{K}(x,t;y,0)dy\leq C\|f_0\|_{C^{0,1}}.
			\end{split}
		\end{equation}
		On the other hands, $f$ can be viewed as a solution of the non-homogeneous heat equation
		\begin{equation}
			\begin{split}
				f_t-\Delta_{M_t}f=|A|^{2}(x,t)f
			\end{split}
		\end{equation}
		with initial data $f(\cdot,0)=f_0$. So, we have 
		\begin{equation}
			\begin{split}
				f(x,t)=\int_0^t\int_{M_s}\tilde{G}(x,t;y,s)|A|^{2}(y,s)f(y,s)dyds+\int_{M_t}\tilde{G}(x,t;y,0)f_0(y)dy
			\end{split}
		\end{equation}
		and
		\begin{equation}\label{eq4.13}
			\begin{split}
				\|f\|_{X_T}&\leq \bigg\|\int_0^t\int_{M_s}\tilde{G}(x,t;y,s)|A|^{2}fdyds\bigg\|_{X_T}\\
				&+\bigg\|\int_{M_t}\tilde{G}(x,t;y,0)f_0dy\bigg\|_{X_T}
			\end{split}
		\end{equation}
		Similar to Proposition \ref{prop4.1},
		\begin{equation}\label{eq4.14}
			\begin{split}
				&\bigg\|\int_0^t\int_{M_s}\tilde{G}(x,t;y,s)|A|^{2}fdyds\bigg\|_{X_T}\\
				&\quad\leq C\||A|^{2}f\|_{Y_T}\\
				&\quad\leq Ck^2_0\sup\limits_{x\in M}\sup\limits_{0<r^2<T}r^{\frac{2}{n+4}}\|f\|_{L^{n+4}(\Omega(x,r))}\\
				&\quad\leq C\|f_0\|_{L^\infty}.
			\end{split}
		\end{equation}
		where in the last equation, we use \eqref{eq4.5}. So it remains to prove
		\begin{equation}\label{eq4.15}
			\begin{split}
				\bigg\|\int_{M_t}\tilde{G}(x,t;y,0)f_0dy\bigg\|_{X_T}\leq C\|f_0\|_{C^{0,1}}.
			\end{split}
		\end{equation}
		
		Denote
		\begin{equation}
			\tilde{f}=\int_{M_t}\tilde{G}(x,t;y,0)f_0dy
		\end{equation}
		We have
		\begin{equation}\label{eq5.18}
			\partial_t\tilde{f}-\Delta_{x,t}\tilde{f}=0.
		\end{equation}
		
		\textbf{ Step 1. Estimate $\|\tilde{f}\|_{L^\infty}$:}
		\begin{equation}
			|\tilde{f}(x,t)|=\bigg|\int_{M_t}\tilde{G}(x,t;y,0)f_0dy\bigg|\leq \|f_0\|_{L^\infty}\int_{M_t}\tilde{G}(x,t;y,0)dy \leq C\|f_0\|_{L^\infty}.
		\end{equation}
		
		\textbf{ Step 2. Estimate $\|\nabla\tilde{f}\|_{L^\infty}$:}
		
		We consider the evolving equation of $\nabla \tilde{f}$. Taking derivative of \eqref{eq5.18},
		\begin{equation}\label{eq5.20}
			\nabla\partial_t\tilde{f}-\nabla\Delta_{x,t}\tilde{f}=0.
		\end{equation}
		Noting that along the mean curvature flow
		\begin{equation}\label{eq5.21}
			\nabla (\partial_t \tilde{f}) - \partial_t (\nabla \tilde{f}) = H h_{ik} g^{kj} \nabla_j \tilde{f}\tau_i
		\end{equation}
		where $\{\tau_i\}_{i=1}^n$ is an orthogonal basis of $T_xM$.
		By Bochner's formula, we have
		\begin{equation}\label{eq5.22}
			\nabla \Delta \tilde{f} - \Delta \nabla \tilde{f} = \text{Ric}(\tilde{f}, \cdot)^\sharp.
		\end{equation}
		Plugging \eqref{eq5.21} and \eqref{eq5.22} into \eqref{eq5.20}, we obtain 
		\begin{equation}
			\begin{split}
				\partial_t (\nabla \tilde{f}) - \Delta \nabla \tilde{f}=\text{Ric}(\nabla \tilde{f}, \cdot)^\sharp-H h_{ik} g^{kj} \nabla_j \tilde{f}\tau_i.
			\end{split}
		\end{equation}
		By direct calculation and Kato's inequality, we arrive at 
		\begin{equation}
			\partial_t |\nabla \tilde{f}| - \Delta |\nabla \tilde{f}|\leq C\kappa^2|\nabla \tilde{f}|
		\end{equation}
		in distribution sense.
		At $t=0$, $|\nabla\tilde{f}(\cdot,0)|=|\nabla f_0|\leq C\|f_0\|_{C^{0,1}}$. By Moser's Iteration
		\begin{equation}
			|\nabla \tilde{f}(x,t)|\leq e^{C\kappa^2 t}\int_M\tilde{G}(x,t;y,0)|\nabla f_0|\leq C\|f_0\|_{C^{0,1}}
		\end{equation}
		for any $0<t\leq 1$.

		\textbf{ Step 3. Estimate $\|\nabla^2\tilde{f}\|_{L^{n+4}}$:}
		
		We write the evolving equation of $\nabla \tilde{f}$ in $W^{2,n+4}$ coordinate,
		\begin{equation}
			\frac{\partial \tilde{f}_i}{\partial t} - g^{jk} \partial_j \partial_k \tilde{f}_i = R_{i}^{\;p} \tilde{f}_p - H h_{i}^{\;p} \tilde{f}_p.
		\end{equation}
		where $\tilde{f}_i=\partial_i \tilde{f}$. By Calder\'on-Zygmund inequality, we have
		\begin{equation}
			r^{\frac{2}{n+4}} \|\partial\partial \tilde{f}\|_{L^{n+4}(\Omega(x,t))}\leq C\|\partial \tilde{f}\|_{L^\infty}.
		\end{equation}
		On the other hand, 
		\begin{equation}
			\nabla \tilde{f}=\partial_k\tilde{f} dx^k\quad \text{and}\quad \nabla^2 \tilde{f}=\bigg(\partial_i\partial_j\tilde{f}-\Gamma_{ij}^k\partial_k\tilde{f}\bigg)dx^i\otimes dx^j.
		\end{equation}
		So, we have
		\begin{equation}
			r^{\frac{2}{n+4}} \|\nabla^2\tilde{f}\|_{L^{n+4}(\Omega(x,t))}\leq C\|\partial \tilde{f}\|_{L^\infty}\leq C\|f_0\|_{C^{0,1}}.
		\end{equation}
	\end{proof}

	\begin{theorem}\label{thm4.4}
		Let $(M_t)_{t\in [0,T]}$ be a mean curvature flow. Suppose that $M_0$ has bounded second fundamental form, i.e., $\|A\|_{M_0}\leq \kappa$. There exist $T'=T'(n,\kappa,T)$,  $\varepsilon=\varepsilon(n,\kappa,T)$ and $C_7=C_7(n,\kappa,T)$  with following property.  For any  function $u_0$ such that $\|u_0\|_{C^{0,1}(M_0)}\leq \varepsilon$,  there exists a solution $u(\cdot,t)\in X_{T'}$ which resolve \eqref{eq1.8}:
		\begin{equation}
			\begin{split}
				u(x,t)&:=\int_M\Tilde{K}(x,t;y,0)u_0dy\\
				&\quad+\int_{M\times [0,t]}\tilde{K}(x,t;y,s)\mathcal{Q}(u)(y,s)dyds.
			\end{split}
		\end{equation}
		Furthermore,  for any $k\in \IN_0$ and every multiindex $\alpha\in \IN_0^n$ we have the estimates
		\begin{equation}\label{eq4.28}
			\sup\limits_{x\in M}\sup\limits_{t>0}|(t^{\frac{1}{2}}\nabla)^\alpha(t\partial_t)^k\nabla u|\leq C\|u_0\|_{C^{0,1}}.
		\end{equation}
	\end{theorem}
	
	\begin{proof}
		We take $T'$ sufficiently small so that for Proposition \ref{prop4.1}, \ref{prop4.2} and \ref{prop4.3} hold.
		Denote $X^\delta_{T'}:=\{u\in X_{T'}:\|u\|_{X_{T'}}\leq \delta,u(\cdot,0)=u_0\}$.	Let $\mathcal{G}:X^\delta_{T'}\to X^\delta_{T'}$ be a map defined by
		\begin{equation}
			\begin{split}
				\mathcal{G}(u)&:=\int_M\tilde{K}(x,t;y,0)u_0dy\\
				&+\int_{M\times [0,t]}\tilde{K}(x,t;y,s)\mathcal{Q}(u)(y,s)dyds.
			\end{split}
		\end{equation}
		For any $u_1,u_2\in X^\delta_{T'}$, by Proposition \ref{prop4.1}, we have 
		\begin{equation}
			\begin{split}
				\|\mathcal{G}(u_1)-\mathcal{G}(u_2)\|_{X_{T'}}&=\bigg\|\int_{M\times [0,t]}\tilde{K}(x,t;y,s)(\mathcal{Q}(u_1)-\mathcal{Q}(u_2))dyds\bigg\|_{X_{T'}}\\
				&\leq C_4\|\mathcal{Q}(u_1)-\mathcal{Q}(u_2)\|_{Y_{T'}}.
			\end{split}
		\end{equation}
		Using Proposition \ref{prop4.2}, by taking $\varepsilon$ small enough, we have
		\begin{equation}\label{c2-1}
			\begin{split}
				\|\mathcal{G}(u_1)-\mathcal{G}(u_2)\|_{X_{T'}}&\leq C_4C_5(\|u_1\|_{X_{T'}}+\|u_2\|_{X_{T'}})\|u_1-u_2\|_{X_{T'}}\\
				&\leq 2C_4C_5\delta\|u_1-u_2\|_{X_{T'}}.
			\end{split}
		\end{equation}
		On the other hand, by Proposition \ref{prop4.3}, we obtain 
		\begin{equation}
			\begin{split}
				\|\mathcal{G}(u)\|_{X_{T'}}&\leq \bigg\|\int_{M}\tilde{K}(x,t;y,0)u_0dy\bigg\|_{X_{T'}}\\
				&+\bigg\|\int_{M\times [0,t]}\tilde{K}(x,t;y,s)\mathcal{Q}(u)dyds\bigg\|_{X_{T'}}\\
				&\leq C_5\|u\|^2_{X_{T'}}+C_6\|u_0\|_{C^{0,1}}.
			\end{split}
		\end{equation}
		If we take $\delta=\frac{1}{4C_4C_5}$, and $\varepsilon=\frac{1}{4C_4C_5C_6}$, we get
		\begin{equation}\label{eq4.33}
			\|\mathcal{G}(u_1)-\mathcal{G}(u_2)\|_{X_{T'}}\leq \frac12 \|u_1-u_2\|_{X_{T'}}
		\end{equation}
		and 
		\begin{equation}\label{eq4.34}
			\|\mathcal{G}(u)\|_{X_{T'}}\leq \delta.
		\end{equation}
		The theorem follows from \eqref{eq4.33}, \eqref{eq4.34}, and contraction mapping. And the inequality \eqref{eq4.28} follows from standard estimates of parabolic equation.
	\end{proof}
	
	\begin{proof}[Proof of Theorem \ref{thm0.2}:]
		By Proposition \ref{prop1.5}, Theorem \ref{thm0.2} follows from Theorem \ref{thm4.4}.
	\end{proof}

\end{document}